\pdfoutput=1
\documentclass[reqno, 12pt]{amsart} 
\def\ssign{\textsection\nobreak\hspace{1pt plus 0.3pt}}
\makeatletter
\let\origsection=\section 
\def\mysection{\@mystartsection{section}{1}\z@{.7\linespacing\@plus\linespacing}{.5\linespacing}{\normalfont\scshape\centering\ssign}}
\def\section{\@ifstar{\origsection*}{\mysection}}
\def\appendix{\par\c@section\z@ \c@subsection\z@
   \let\sectionname\appendixname
   \let\section=\origsection  
   \def\thesection{\@Alph\c@section}}
\def\@mystartsection#1#2#3#4#5#6{\if@noskipsec \leavevmode \fi
 \par \@tempskipa #4\relax
 \@afterindenttrue
 \ifdim \@tempskipa <\z@ \@tempskipa -\@tempskipa \@afterindentfalse\fi
 \if@nobreak \everypar{}\else
     \addpenalty\@secpenalty\addvspace\@tempskipa\fi
 \@dblarg{\@mysect{#1}{#2}{#3}{#4}{#5}{#6}}}
\def\@mysect#1#2#3#4#5#6[#7]#8{\edef\@toclevel{\ifnum#2=\@m 0\else\number#2\fi}\ifnum #2>\c@secnumdepth \let\@secnumber\@empty
  \else \@xp\let\@xp\@secnumber\csname the#1\endcsname\fi
  \@tempskipa #5\relax
  \ifnum #2>\c@secnumdepth
    \let\@svsec\@empty
  \else
    \refstepcounter{#1}\edef\@secnumpunct{\ifdim\@tempskipa>\z@ \@ifnotempty{#8}{\@nx\enspace}\else
        \@ifempty{#8}{.}{\@nx\enspace}\fi
    }\@ifempty{#8}{\ifnum #2=\tw@ \def\@secnumfont{\bfseries}\fi}{}\protected@edef\@svsec{\ifnum#2<\@m
        \@ifundefined{#1name}{}{\ignorespaces\csname #1name\endcsname\space
        }\fi
      \@seccntformat{#1}}\fi
  \ifdim \@tempskipa>\z@ \begingroup #6\relax
    \@hangfrom{\hskip #3\relax\@svsec}{\interlinepenalty\@M #8\par}\endgroup
    \ifnum#2>\@m \else \@tocwrite{#1}{#8}\fi
  \else
  \def\@svsechd{#6\hskip #3\@svsec
    \@ifnotempty{#8}{\ignorespaces#8\unskip
       \@addpunct.}\ifnum#2>\@m \else \@tocwrite{#1}{#8}\fi
  }\fi
  \global\@nobreaktrue
  \@xsect{#5}}
\makeatother

\usepackage{amsmath,amssymb,amsthm}
\usepackage{mathrsfs}
\usepackage{mathabx}\changenotsign
\usepackage{dsfont}

\usepackage{xcolor}
\usepackage[backref]{hyperref}
\hypersetup{
    colorlinks,
    linkcolor={red!60!black},
    citecolor={green!60!black},
    urlcolor={blue!60!black}
}

\usepackage[open,openlevel=2,atend]{bookmark}

\usepackage[abbrev,msc-links,backrefs]{amsrefs}
\BibSpec{book}{+{} {\PrintPrimary} {transition}
	+{,} { \textit} {title}
	+{.} { } {part}
	+{:} { \textit} {subtitle}
	+{,} { \PrintEdition} {edition}
	+{} { \PrintEditorsB} {editor}
	+{,} { \PrintTranslatorsC} {translator}
	+{,} { \PrintContributions} {contribution}
	+{,} { } {series}
	+{,} { \voltext} {volume}
	+{,} { } {publisher}
	+{,} { } {organization}
	+{,} { } {address}
	+{,} { \PrintDateB} {date}
	+{,} { } {status}
	+{,} { \PrintDOI} {doi}
	+{} { \parenthesize} {language}
	+{} { \PrintTranslation} {translation}
	+{;} { \PrintReprint} {reprint}
	+{.} { } {note}
	+{.} {} {transition}
	+{} {\SentenceSpace \PrintReviews} {review}
}

\usepackage{doi}

\renewcommand{\PrintDOI}[1]{\doi{#1}}

\usepackage[T1]{fontenc}
\usepackage{lmodern}
\usepackage[babel]{microtype}
\usepackage[english]{babel}

\usepackage{geometry}
\numberwithin{equation}{section}
\numberwithin{figure}{section}

\usepackage{enumitem}
\def\rmlabel{\upshape({\itshape \roman*\,})}

\def\alabel{\upshape({\itshape \alph*\,})}

\def\nlabel{\upshape({\itshape \arabic*\,})}

\theoremstyle{plain}
\newtheorem{thm}{Theorem}[section]
\newtheorem{fact}[thm]{Fact}
\newtheorem{prop}[thm]{Proposition}
\newtheorem{clm}[thm]{Claim}

\theoremstyle{definition}

\newtheorem{dfn}[thm]{Definition}

\let\theta=\vartheta
\let\rho=\varrho
\let\phi=\varphi

\def\NN{\mathds N}

\def\ccH{{\mathscr{H}}}

\def\ccN{\mathscr{N}}
\def\ccX{\mathscr{X}}
  
\let\polishlcross=\l
\def\l{\ifmmode\ell\else\polishlcross\fi}

\makeatletter
\def\moverlay{\mathpalette\mov@rlay}
\def\mov@rlay#1#2{\leavevmode\vtop{   \baselineskip\z@skip \lineskiplimit-\maxdimen
   \ialign{\hfil$\m@th#1##$\hfil\cr#2\crcr}}}
\newcommand{\charfusion}[3][\mathord]{
    #1{\ifx#1\mathop\vphantom{#2}\fi
        \mathpalette\mov@rlay{#2\cr#3}
      }
    \ifx#1\mathop\expandafter\displaylimits\fi}
\makeatother

\newcommand{\dcup}{\charfusion[\mathbin]{\cup}{\cdot}}

\def\tand{\ \text{and}\ }
\def\qand{\quad\text{and}\quad}

\newcommand{\vrhup}[1]{\scaleobj{0.6}{\scalerel*{\rightharpoonup}{#1}}}
\newcommand{\nrhup}{\mathord{\scaleobj{0.6}{\scalerel*{\rightharpoonup}{x}}}}
\newcommand{\wrhup}{\scaleobj{0.6}{\scalerel*{\rightharpoonup}{W}}}
\def\vseq#1{\ThisStyle{  \mathord{\vbox{\offinterlineskip\ialign{    \hfil##\hfil\cr
    $\SavedStyle{}_{\smash{\vrhup#1}}$\cr
    \noalign{\kern-0.7\scriptspace}
    $\SavedStyle#1$\cr}}}}}
\def\seq#1{\ThisStyle{  \mathord{\vbox{\offinterlineskip\ialign{    \hfil##\hfil\cr
    $\SavedStyle{}_{\smash{\nrhup}}$\cr
    \noalign{\kern-0.5\scriptspace}
    $\SavedStyle#1$\cr}}}}}
\def\wseq#1{\ThisStyle{  \mathord{\vbox{\offinterlineskip\ialign{    \hfil##\hfil\cr
    $\SavedStyle{}_{\smash{\wrhup#1}}$\cr
    \noalign{\kern-0.7\scriptspace}
    $\SavedStyle#1$\cr}}}}}

\let\lra=\longrightarrow
\let\to=\lra

\makeatletter
\newcommand{\pushright}[1]{\ifmeasuring@#1\else\omit\hfill$\displaystyle#1$\fi\ignorespaces}
\newcommand{\pushleft}[1]{\ifmeasuring@#1\else\omit$\displaystyle#1$\hfill\fi\ignorespaces}
\makeatother

\usepackage{tikz}
\usetikzlibrary{calc}

\newcommand{\hedge}[7]{

	\ifx\relax#4\relax
		\def\qoffs{0pt}
	\else
		\def\qoffs{#4}
	\fi

	\def\qhedge{
		($#1+#3!\qoffs!-90:#2-#3$)--
		($#2+#1!\qoffs!-90:#3-#1$)--
		($#3+#2!\qoffs!-90:#1-#2$)--cycle}

	\coordinate (12) at ($#1!\qoffs!90:#2$);
	\coordinate (13) at ($#1!\qoffs!-90:#3$);
	\coordinate (23) at ($#2!\qoffs!90:#3$);
	\coordinate (21) at ($#2!\qoffs!-90:#1$);
	\coordinate (31) at ($#3!\qoffs!90:#1$);
	\coordinate (32) at ($#3!\qoffs!-90:#2$);
	
	\def\nqhedge{
		(13) let \p1=($(13)-#1$), \p2=($(12)-#1$) in
			arc[start angle={atan2(\y1,\x1)}, delta angle={atan2(\y2,\x2)-atan2(\y1,\x1)-360*(atan2(\y2,\x2)-atan2(\y1,\x1)>0)}, x radius=\qoffs, y radius=\qoffs]--
		(21) let \p1=($(21)-#2$), \p2=($(23)-#2$) in
			arc[start angle={atan2(\y1,\x1)}, delta angle={atan2(\y2,\x2)-atan2(\y1,\x1)-360*(atan2(\y2,\x2)-atan2(\y1,\x1)>0)}, x radius=\qoffs, y radius=\qoffs]--
		(32) let \p1=($(32)-#3$), \p2=($(31)-#3$) in
			arc[start angle={atan2(\y1,\x1)}, delta angle={atan2(\y2,\x2)-atan2(\y1,\x1)-360*(atan2(\y2,\x2)-atan2(\y1,\x1)>0)}, x radius=\qoffs, y radius=\qoffs]--cycle}

		\ifx\relax#5\relax
		\def\qlwidth{1pt}
	\else
		\def\qlwidth{#5}
	\fi
	
		\ifx\relax#7\relax
		\fill \nqhedge;
	\else
		\fill[#7]\nqhedge;
	\fi

		\ifx\relax#6\relax
		\draw[line width=\qlwidth,rounded corners=\qoffs]\nqhedge;
	\else
		\draw[line width=\qlwidth,#6]\nqhedge;
	\fi
}

\usepackage{subcaption}
\begin{document}
\title[Ramsey-type problems for generalised Sidon sets]{Ramsey-type problems for generalised Sidon sets}

\dedicatory{Dedicated to the memory of Tomasz Schoen}

\author[Chr. Reiher]{Christian Reiher}
\address{Fachbereich Mathematik, Universit\"at Hamburg, Hamburg, Germany}
\email{christian.reiher@uni-hamburg.de}

\author[V. R\"{o}dl]{Vojt\v{e}ch R\"{o}dl}
\address{Department of Mathematics, Emory University, Atlanta, USA}
\email{vrodl@emory.edu}

\author[M. Schacht]{Mathias Schacht}
\address{Fachbereich Mathematik, Universit\"at Hamburg, Hamburg, Germany}
\email{schacht@math.uni-hamburg.de}

\thanks{The second author is supported by NSF grant DMS~2300347.}


\begin{abstract}
	We establish the existence of generalised Sidon sets enjoying additional Ramsey-type properties, which are motivated by questions of Erd\H os and Newman
	and of Alon and Erd\H os.
\end{abstract}

\maketitle


\section{Introduction}
\label{sec:introduction}
In December 2024, we received the sad news of the sudden demise of \textsc{Tomasz Schoen}. Tomek was a brilliant mathematician 
who lived up to his beautiful surname in his research. He made many important contributions to the field of additive number theory, and 
a large part of his research was focused on density problems for arithmetic structures related to the theorems of Roth and Szemer\'edi. 
Among other topics, he was also interested in Ramsey-type questions for the integers.

For a subset $X\subseteq \NN$ and a fixed integer $k\geq 2$, we denote by
$\rho_{X,k}\colon \NN\to\NN$ the function of the number of additive representations with $k$ terms from~$X$,
i.e., for every $n\in\NN$, we set
\[
	\rho_{X,k}(n)=\big|\big\{(x_1,\dots,x_k)\in X^k\colon x_1+\dots +x_k=n \tand x_1\leq \dots\leq x_k\big\}\big|\,.
\]
Moreover, we define
\[
	\rho_{k}(X)=\sup\{\rho_{X,k}(n)\colon n\in\NN\}
\]
and we say $X$ is a \emph{$B_{k,\l}$-set}, if $\rho_{k}(X)=\l$, i.e., if at least one integer has $\l$ ordered $k$-term representations
in $X$ and no integer has more.

The study of $B_{k,\l}$-sets can be traced back to the work of Sidon~\cite{Si32} (see, e.g., the textbook of Halberstam and Roth~\cite{HR83}*{Chapter~II})
and $B_{2,1}$-sets
are usually referred to as \emph{Sidon sets} (or \emph{Sidon sequences}). In general $B_{k,1}$-sets
have the property that all $k$-term sums are distinct and for simplicity we refer to those sets as \emph{$B_k$-sets}.

We are interested in the existence of infinite $B_{k,\l}$-sets with special properties motivated by questions of
Erd\H os  and Newman (see, e.g.,~\cite{E80}) and by Alon and Erd\H os~\cite{AE85}. Below we briefly discuss
these properties and our main result, Theorem~\ref{thm:B-set},
asserts the existence of~$B_{k,\l}$-sets enjoying all those qualities.

\subsection{Ramsey-type property}\label{sec:EN}
Our starting point is a question independently proposed by Erd\H os and
Newman. We shall employ the arrow notation from Ramsey theory and write
\[
	X\lra [k,\l]_r\,,
\]
to signify the statement that for every $r$-colouring $X=X_1\dcup\dots\dcup X_r$ there is some colour class $X_q$ such that
\[
	\rho_{k}(X_q)=\l\,.
\]
In other words, for some colour class $X_q$
there are $\l$ distinct $k$-tuples $(x_i^{(j)})_{i=1}^k\in X^k_q$ with $x_1^{(j)}\leq \dots \leq x_k^{(j)}$ for $j=1,\dots,\l$ such that
\[
	\sum_{i=1}^kx_i^{(1)}=\sum_{i=1}^kx_i^{(2)}=\dots=\sum_{i=1}^kx_i^{(\l)}\,.
\]

We note that,
if $X\lra [k,\l]_r$ holds for every $r\geq 2$, then for every finite colouring of~$X$ there still exists some integer
which can be represented in $\l$ different ways as a sum of~$k$ terms from the same colour class. It might seem plausible, that
for such a partition relation to hold, the set $X$ itself may have to represent some integers in more than $\l$ ways. However, it turned out that
this na\"\i ve idea is false. In fact,  Erd\H os~\cite{E80}
established the existence of an infinite $B_{2,3}$-set $X$ satisfying
\[
	X\lra [2,3]_r
\]
for every $r\geq 2$. This means, that even though no integer can be additively
represented in four different ways by two elements of $X$, for every finite colouring of $X$ there is still
some integer enjoying three different $2$-term representations with all terms from the same colour class. Formulated this way,
this result has some resemblance to Folkman-type results in Ramsey theory for graphs (cf.\ the work of Folkman~\cite{Fo70} and the textbook of Graham, Spencer, and Rothschild~\cite{GRS90}*{\ssign5.3}).
Erd\H os conjectured, that such $B_{2,\l}$-sets exist for every $\l\geq 2$ (and he proved it for $\l$ being a power of $2$ and for $\l$
of the form $\frac{1}{2}\binom{2s}{s}$ for some $s\geq 1$). This conjecture was addressed by Ne\v set\v ril and R\"odl~\cite{NR85}
and here we build on their work and extend it to
$B_{k,\l}$-sets for arbitrary~$k>2$ (see property~\ref{iti:B-set-EN} in Theorem~\ref{thm:B-set} below).

\subsection{Local-global structure}\label{sec:AE}
The second property under consideration is motivated by a question of Alon and Erd\H os~\cite{AE85} for Sidon sets,
inspired by Pisier's problem~\cite{Pi83}.  Alon and Erd\H os asked if the following property characterises sets $X$ that are a
finite union of Sidon sets:
\begin{equation}\label{eq:PSidon}
	\textit{For some $\delta>0$ every finite $Y\subseteq X$ contains a Sidon set of size at least $\delta\,|Y|$.}
\end{equation}
It is immediate that the local property~\eqref{eq:PSidon} holds for every set $X$ being a finite union of Sidon sets.
In fact, this implication is true for every hereditary property and not specific to the property of Sidon sets. Hence, the question of
Alon and Erd\H os asks if the local property~\eqref{eq:PSidon} implies that $X$ has a `simple' global
structure, rendered by being decomposable into a finite number of Sidon sets.

There are a few instances where indeed the global structure has such a local characterisation. Most notably, Horn~\cite{Ho55} established such an assertion
for independent sets in vector spaces and, more generally, Edmonds~\cite{Ed65} proved such a statement for independent sets in matroids.
However, in the context of Sidon sets Ne\v set\v ril, R\"odl, and Sales~\cite{NRS}
(partly based on earlier work with Erd\H os) showed that such a local characterisation fails (see also~\cite{RRS}*{\ssign5}).
We strengthen their result and show that the implication fails already for $\delta=1/4$
(see properties~\ref{iti:B-set-EN} and~\ref{iti:B-set-AE} for $k=2$ and any $\l\geq 2$ in Theorem~\ref{thm:B-set} below).

\subsection{Main result}
Besides the properties discussed so far, our method yields $B_{k,\l}$-sets~$X$ enjoying additional properties (see properties~\ref{iti:Bh}\,--\,\ref{iti:distinct}
in Theorem~\ref{thm:B-set} below). Properties~\ref{iti:Bh} and~\ref{iti:Bij} imply that there is at most one additive representation for any integer
with at most $k$ terms from $X$. Consequently, assertions~\ref{iti:B-set-EN} and~\ref{iti:B-set-AE} stay valid,
even if we relax the definition of $B_{k,\l}$-sets by considering all additive representations with \emph{up to at most} $k$ terms.
Part~\ref{iti:2tol} tells us that every integer has zero, one or $\l$ additive representations from $X$ and the last property asserts that all terms appearing in two different $k$-term representations are distinct.

\begin{thm}\label{thm:B-set}
	For all integers $k\geq 2$, $\l\geq 2$, there exists an infinite $B_{k,\l}$-set $X\subseteq \NN$ satisfying the following properties:
	\begin{enumerate}[label=\rmlabel]
		\item\label{iti:B-set-EN}
		      For every integer $r\geq 2$ we have $X\lra[k,\l]_r$.
		\item\label{iti:B-set-AE}
		      Every finite subset $Y\subseteq X$ contains a $B_k$-set of size at least $\tfrac{k-1}{2k}|Y|$.
	\end{enumerate}
	In addition $X$ also satisfies:
	\begin{enumerate}[label=\rmlabel,resume]
		\item\label{iti:Bh}
		      For every $h=2,\dots,k-1$ the set $X$ is a $B_h$-set.
		\item\label{iti:Bij}
		      For every $1\leq i<j$ with $i+j<2k$ we have $\rho_{X,i}(n)+\rho_{X,j}(n)\leq 1$.
		\item\label{iti:2tol}
		      If $\rho_{X,k}(n)\geq 2$ for some $n\in\NN$, then $\rho_{X,k}(n)=\l$.
		\item\label{iti:distinct}
		      If $x_1+\dots+x_k=y_1+\dots+y_k$ for some $x_1\leq\dots\leq x_k$, $y_1\leq\dots\leq y_k$ from~$X$,
		      then either $x_i=y_i$ for all $i\in[k]$ or all $2k$ terms $x_1,\dots,x_k,y_1,\dots,y_k$ are distinct.
	\end{enumerate}
\end{thm}

Following the idea from the work of Ne\v set\v ril and R\"odl~\cite{NR85},
for $\l=2$ the proof of Theorem~\ref{thm:B-set} relies on the existence of Ramsey graphs for the cycle $C_{2k}$
having girth $2k$ themselves. Such a result is a consequence of the work of R\"odl and Ruci\'nski~\cite{RR95} on thresholds for Ramsey properties in random graphs.
For $\l>2$ the cycle~$C_{2k}$ is replaced by so-called generalised theta graphs $\Theta_{k,\l}$ and the existence of a Ramsey graph of the same girth
follows from recent work of the first two authors~\cite{RR}. For the proof of Theorem~\ref{thm:B-set} we shall employ a version of these results
for ordered graphs. In the next section we introduce these results from Ramsey theory.

We close this introduction by remarking that we leave it open if the constant factor $\frac{k-1}{2k}$ in 
clause~\ref{iti:B-set-AE} of Theorem~\ref{thm:B-set} is best possible. The proof of the theorem in \ssign\ref{sec:proof} 
relies on finding large $\Theta_{k,\l}$-free ordered subgraphs and the recent work in~\cite{RRSS} shows that for this general 
Tur\'an-type problem the constant is optimal.
%
%

\section{Local Ramsey theory}
All graphs in this article have ordered vertex sets and all graph isomorphisms respect
the orderings of the vertex sets. So for any two graphs there is at most one isomorphism
between them. Given two integers $k, \ell\ge 2$ we fix a {\it generalised theta
		graph} $\Theta_{k, \ell}$ consisting of $\ell$ internally vertex disjoint, ascending
paths of length $k$ with common end points. Thus the graph~$\Theta_{k, \ell}$ has
\begin{enumerate}
	\item[$\bullet$] $(k-1)\ell+2$ vertices, say $a_0$, $a_k$,
	      and $a_i^{(j)}$ with $(i, j)\in [k-1]\times [\ell]$,
	\item[$\bullet$] and $k\ell$ edges so that $a_0a^{(j)}_1\dots a^{(j)}_{k-1}a_k$
	      is a path for every $j\in [\ell]$ (see Figure~\ref{fig:theta}).
\end{enumerate}
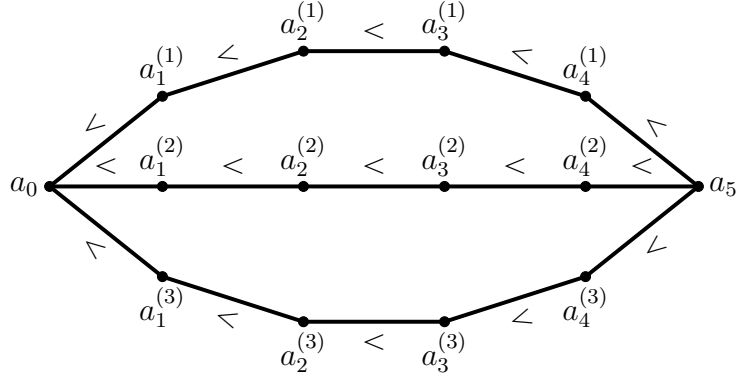
\begin{figure}[ht]
	\centering
	\begin{tikzpicture}[scale=0.75]

		\coordinate (x0) at (0.5,0);

		\coordinate (x1) at (2.5,1.6);
		\coordinate (x2) at (5,2.4);
		\coordinate (x3) at (7.5,2.4);
		\coordinate (x4) at (10,1.6);

		\coordinate (x5) at (2.5,0);
		\coordinate (x6) at (5,0);
		\coordinate (x7) at (7.5,0);
		\coordinate (x8) at (10,0);

		\coordinate (x9) at (2.5,-1.6);
		\coordinate (x10) at (5,-2.4);
		\coordinate (x11) at (7.5,-2.4);
		\coordinate (x12) at (10,-1.6);

		\coordinate (x13) at (12,0);

		\draw[line width=1.5pt]
		(x0) -- node[above,sloped] {$<$}
		(x1) -- node[above,sloped] {$<$}
		(x2) -- node[above,sloped] {$<$}
		(x3) -- node[above,sloped] {$<$}
		(x4) -- node[above,sloped] {$<$}
		(x13);
		\draw[line width=1.5pt]
		(x0) -- node[above,sloped] {$<$}
		(x5) -- node[above,sloped] {$<$}
		(x6) -- node[above,sloped] {$<$}
		(x7) -- node[above,sloped] {$<$}
		(x8) -- node[above,sloped] {$<$}
		(x13);
		\draw[line width=1.5pt]
		(x0) -- node[below,sloped] {$<$}
		(x9) -- node[below,sloped] {$<$}
		(x10) -- node[below,sloped] {$<$}
		(x11) -- node[below,sloped] {$<$}
		(x12) -- node[below,sloped] {$<$}
		(x13);

		\foreach \i in {0,...,13}
		\filldraw (x\i) circle (2.5pt);

		\node[anchor=east] at (x0) {$a_0$};
		\node[anchor=west] at (x13) {$a_5$};

		\node[anchor=south] at (x1) {$a_1^{(1)}$};
		\node[anchor=south] at (x2) {$a_2^{(1)}$};
		\node[anchor=south] at (x3) {$a_3^{(1)}$};
		\node[anchor=south] at (x4) {$a_4^{(1)}$};

		\node[anchor=south] at (x5) {$a_1^{(2)}$};
		\node[anchor=south] at (x6) {$a_2^{(2)}$};
		\node[anchor=south] at (x7) {$a_3^{(2)}$};
		\node[anchor=south] at (x8) {$a_4^{(2)}$};

		\node[anchor=north] at (x9) {$a_1^{(3)}$};
		\node[anchor=north] at (x10) {$a_2^{(3)}$};
		\node[anchor=north] at (x11) {$a_3^{(3)}$};
		\node[anchor=north] at (x12) {$a_4^{(3)}$};
	\end{tikzpicture}
	\caption{Ordered theta graph $\Theta_{5,3}$.}
	\label{fig:theta}
\end{figure}
The demand that these path be ascending means that we require
$a_0<a^{(j)}_1<\dots<a^{(j)}_{k-1}<a_k$
for every $j\in [\ell]$. As it is immaterial to our argument how internal
vertices of distinct paths compare to each other with respect to the vertex ordering
of $\Theta_{k, \ell}$, we would like to leave this unspecified.

Roughly speaking, the reason why theta graphs help us in our investigation of
$B_{k, \ell}$-sets is that when $V(\Theta_{k, \ell})\subseteq\NN$
and
\[
	x^{(j)}_i=a^{(j)}_i-a^{(j)}_{i-1}\,,
\]
where $a^{(j)}_0=a_0$ and $a^{(j)}_k=a_k$,
then the~$\ell$ sums $\sum_{i=1}^k x^{(j)}_i$ telescope to $a_k-a_0$ and, hence,
they are equal. The set $X$ promised by Theorem~\ref{thm:B-set} will be derived
from an appropriate (infinite) graph $G$ with $V(G)\subseteq \NN$ by setting
\[
	X=\{b-a\colon ab\in E(G)\text{ and } a<b\}\,.
\]
Thus all occurrences
of~$\Theta_{k, \ell}$ in $G$ correspond to numbers expressible in $\ell$ distinct ways
as~$k$-term sums with elements from $X$. Moreover, the partition relation $X\lra[k, \ell]_r$ can be
enforced by letting~$G$ contain a Ramsey graph of $\Theta_{k, \ell}$ for $r$ colours.

In view of the desired properties of $X$, we will also need $G$ to have some further qualities, such as containing
neither~$\Theta_{k, \ell+1}$ nor cycles of length less than $2k$, but provided that the
natural numbers in $V(G)$ are sufficiently `far apart' it turns out that
all demands on~$X$ translate to such `local' properties of $G$.
The local structure of Ramsey graphs has recently been analysed by the first two authors~\cite{RR},
who established the girth Ramsey theorem.
Here is the precise statement for theta graphs,
which we shall exploit.

\begin{prop}\label{prop:21}
	For all integers $k, \ell\ge 2$ and $s\ge 2k$ there is an infinite
	graph~$G$ with vertex set $V(G)\subseteq\NN$ possessing the following properties:
	\begin{enumerate}[label=\rmlabel]
		\item\label{it:pr1} For every colouring of $E(G)$ with finitely many
		      colours there is a monochromatic induced copy of $\Theta_{k, \ell}$.
		\item\label{it:pr2} Every induced cycle in $G$ of length at most $s$
		      has length exactly $2k$ and belongs to a unique copy of~$\Theta_{k, \ell}$
		      in $G$.
	\end{enumerate}
\end{prop}

For transparency we would like to point out that~\ref{it:pr2} tells us, in particular, that $G$ contains no cycles of length shorter than $2k$ and
how all $2k$-cycles in $G$ are ordered. All of them have two diametrically opposite
vertices joined by two ascending paths of length $k$.
In the remainder of this section we explain how Proposition~\ref{prop:21} follows
from known results~\cite{RR}, while the next section is reserved to the deduction
of Theorem~\ref{thm:B-set} from Proposition~\ref{prop:21}.

We shall utilise a
strong version of the girth Ramsey theorem, whose formulation requires some
preparations.

\begin{dfn}[Forests of copies] \label{dfn:forcop}
	A set of graphs~$\ccN$ is called a {\it forest of copies} if there
	exists an enumeration $\ccN=\{F_1, \dots, F_{|\ccN|}\}$ such that
	for every $j\in [2, |\ccN|]$ the set
	\[
		V(F_j)\cap\bigcup_{i<j}V(F_i)
	\]
	is either empty, a single vertex, or an edge belonging both to~$E(F_j)$
	and to $\bigcup_{i<j}E(F_i)$.
\end{dfn}

For instance, if all graphs in $\ccN$ are single edges, then $\ccN$ is a forest
of copies if and only if these edges form a forest in the ordinary graph theoretical sense.
However, this simple example might create the false impression that forests of copies would be
closed under taking subsets. A counterexample is shown in Figure~\ref{fig:12}.
On the left hand side we see a `cycle of triangles', or more precisely a collection of
five triangles not forming a forest of copies. By adding two further edges, however,
we can hide this configuration inside a forest of eight triangles (see Figure~\ref{fig:12b}).

\begin{figure}[ht]
	\centering
	\begin{subfigure}[b]{0.4\textwidth}
		\centering
		\begin{tikzpicture}

			\coordinate (x0) at (142:1.4cm);
			\coordinate (x1) at (214:1.4cm);
			\coordinate (x2) at (286:1.4cm);
			\coordinate (x3) at (358:1.4cm);
			\coordinate (x4) at (70:1.4cm);
			\coordinate (x5) at (106:2.5cm);
			\coordinate (x6) at (178:2.5cm);
			\coordinate (x7) at (250:2.5cm);
			\coordinate (x8) at (322:2.5cm);
			\coordinate (x9) at (34:2.5cm);

			\hedge{(x0)}{(x1)}{(x6)}{4pt}{1.2pt}{red!80!black}{red!60!white,opacity=0.25}
			\hedge{(x1)}{(x2)}{(x7)}{4pt}{1.2pt}{red!80!black}{red!60!white,opacity=0.25}
			\hedge{(x2)}{(x3)}{(x8)}{4pt}{1.2pt}{red!80!black}{red!60!white,opacity=0.25}
			\hedge{(x3)}{(x4)}{(x9)}{4pt}{1.2pt}{red!80!black}{red!60!white,opacity=0.25}
			\hedge{(x4)}{(x0)}{(x5)}{4pt}{1.2pt}{red!80!black}{red!60!white,opacity=0.25}

			\draw[line width=1pt] (x0)--(x1)--(x2)--(x3)--(x4)--cycle;
			\draw[line width=1pt] (x4)--(x5)--(x0)--(x6)--(x1)--(x7)--(x2)--(x8)--(x3)--(x9)--cycle;

			\foreach \i in {0,...,9}
			\filldraw (x\i) circle (1.7pt);

		\end{tikzpicture}

		\caption{A `cycle' of triangles \dots}
		\label{fig:12a}

	\end{subfigure}
	\hfill
	\begin{subfigure}[b]{0.4\textwidth}
		\centering

		\begin{tikzpicture}

			\coordinate (x0) at (142:1.4cm);
			\coordinate (x1) at (214:1.4cm);
			\coordinate (x2) at (286:1.4cm);
			\coordinate (x3) at (358:1.4cm);
			\coordinate (x4) at (70:1.4cm);
			\coordinate (x5) at (106:2.5cm);
			\coordinate (x6) at (178:2.5cm);
			\coordinate (x7) at (250:2.5cm);
			\coordinate (x8) at (322:2.5cm);
			\coordinate (x9) at (34:2.5cm);

			\hedge{(x0)}{(x1)}{(x6)}{4pt}{1.2pt}{red!80!black}{red!60!white,opacity=0.25}
			\hedge{(x1)}{(x2)}{(x7)}{4pt}{1.2pt}{red!80!black}{red!60!white,opacity=0.25}
			\hedge{(x2)}{(x3)}{(x8)}{4pt}{1.2pt}{red!80!black}{red!60!white,opacity=0.25}
			\hedge{(x3)}{(x4)}{(x9)}{4pt}{1.2pt}{red!80!black}{red!60!white,opacity=0.25}
			\hedge{(x4)}{(x0)}{(x5)}{4pt}{1.2pt}{red!80!black}{red!60!white,opacity=0.25}

			\hedge{(x0)}{(x3)}{(x2)}{4pt}{1.2pt}{red!60!yellow}{red!60!yellow,opacity=0.25}
			\hedge{(x0)}{(x2)}{(x1)}{4pt}{1.2pt}{red!60!yellow}{red!60!yellow,opacity=0.25}
			\hedge{(x0)}{(x4)}{(x3)}{4pt}{1.2pt}{red!60!yellow}{red!60!yellow,opacity=0.25}

			\draw[line width=1pt] (x0)--(x1)--(x2)--(x3)--(x4)--cycle;
			\draw[line width=1pt] (x4)--(x5)--(x0)--(x6)--(x1)--(x7)--(x2)--(x8)--(x3)--(x9)--cycle;
			\draw[line width=1pt] (x3)--(x0)--(x2);

			\foreach \i in {0,...,9}
			\filldraw (x\i) circle (1.7pt);

		\end{tikzpicture}
		\caption{\dots contained in a forest.}
		\label{fig:12b}
	\end{subfigure}
	\caption{Some subforests fail to be forests.}
	\label{fig:12}
\end{figure}
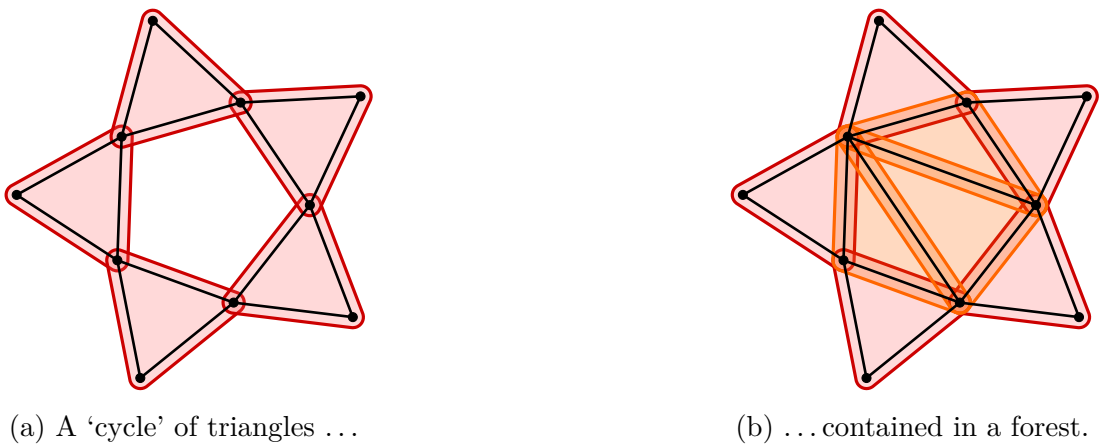

Given two ordered graphs $F$ and~$H$ we write $\binom HF$ for the set of all ordered copies
of~$F$ in~$H$, i.e., for the set of all induced, ordered subgraphs of $H$ isomorphic to $F$.
For a subsystem $\ccH\subseteq \binom HF$ and a number of colours~$r$ the
partition relation $\ccH\lra (F)_r$ indicates that for every colouring
$f\colon E(H)\lra [r]$ one
of the copies in $\ccH$ is monochromatic. The girth Ramsey theorem confirms the
intuition that some such Ramsey systems are locally just forests of copies.
The r\^{o}le of $\ccX$ in the following theorem is due to the fact that, for the above
reason,~$\ccN$ can only be demanded to be a subset of a forest copies rather than an actual forest of copies.

\begin{thm}[Girth Ramsey theorem] \label{thm:grt}
	Given a graph $F$ and $r, n\in \NN$ there exists a graph~$H$
	together with a system of copies $\ccH\subseteq\binom{H}{F}$
	satisfying not only $\ccH\lra (F)_r$, but also the following statement:
	For every $\ccN\subseteq \ccH$ with $|\ccN|\in [2, n]$ there exists
	a set $\ccX\subseteq \ccH$ such that $|\ccX|\le |\ccN|-2$ and $\ccN\cup\ccX$
	is a forest of copies. \qed
\end{thm}

For unordered graphs Theorem~\ref{thm:grt} follows from the work of the first two
authors~\cite{RR}*{Theorems~13.12}.  For ordered graphs Theorem~\ref{thm:grt} can be obtained by the same proof of Theorem~13.12, where for the
Ramsey construction $\Phi$ we appeal to Ramsey's theorem for
ordered graphs (see, e.g.,~\cite{RR}*{\ssign3.5} for details on the partite construction method for ordered graphs).

\begin{proof}[Proof of Proposition~\ref{prop:21}]
	Given integers $k$, $\l$, and $s$ we first observe that the desired
	infinite graph $G$ can be taken as the disjoint union of finite graphs $(H_r)_{r\geq 2}$ satisfying:
	\begin{enumerate}[label=\alabel]
		\item\label{iti:21a} $H_r\lra (\Theta_{k, \ell})_r$,
		\item\label{iti:21b} every induced cycle in $H_r$ of
		      length at most $s$ has length $2k$,
		\item\label{iti:21c} and every such cycle is contained in a unique copy
		      of~$\Theta_{k, \ell}$ in $H_r$.
	\end{enumerate}

	Now given $r$ we appeal to Theorem~\ref{thm:grt} with $F=\Theta_{k, \ell}$
	and $n=\max\{s, k(\ell+1)\}$, thereby obtaining a graph $H$ and a system of
	copies~$\ccH\subseteq \binom H{\Theta_{k, \ell}}$.
	We may assume that $E(H)$ is the union of the edge sets of the
	copies of $\Theta_{k, \ell}$ in~$\ccH$, because throwing away further edges of~$H$ would not destroy
	either of the relevant properties of~$H$ and~$\ccH$. We shall prove that the graph $H_r=H$ displays
	properties~\ref{iti:21a}\,--\,\ref{iti:21c}.

	Part~\ref{iti:21a} immediately follows from  $\ccH\lra (\Theta_{k, \ell})_r$.
	For the proof of assertion~\ref{iti:21b} we consider an induced  cycle~$C$ in~$H$ whose length is at most $s$ and we show that there
	is a copy of~$\Theta_{k, \ell}$ in $\ccH$ containing it. To see this we take a set $\ccN\subseteq \ccH$ of
	size $|\ccN|\le s\le n$ such
	that $\bigcup_{\Theta^{\star}\in\ccN} E(\Theta^{\star})$ contains all edges of~$C$.
	By Theorem~\ref{thm:grt} there exists
	a set $\ccX\subseteq \ccH$ such that $\ccN\cup\ccX$ is a forest of copies.
	Let the enumeration $\ccN\cup\ccX=\{\Theta_1^{\star}, \dots, \Theta_m^{\star}\}$ exemplify this state
	of affairs and let $j\le m$ be minimal such
	that $E(C)\subseteq \bigcup_{i\le j} E(\Theta^{\star}_i)$.

	If $j=1$ then~$C$ is contained
	in $\Theta_1^{\star}$ and we are done. So suppose $j\in [2, m]$ from now on.
	Recall that the set
	\[
		z_j=V(\Theta_j^{\star})\cap\bigcup_{i<j}V(\Theta_i^{\star})
	\]
	satisfies
	either $|z_j|\le 1$
	or $z_j\in E(\Theta_j^{\star})\cap\bigcup_{i<j}E(\Theta_i^{\star})$.

	In the former case the graph $\bigcup_{i\le j}\Theta_i^{\star}$ is the disjoint union or
	a one-point amalgamation of $\bigcup_{i<j}\Theta_i^{\star}$ and $\Theta_j^{\star}$. As these graph
	operations never create new cycles, the minimality of $j$ shows that $C$
	is entirely contained in $\Theta_j^{\star}$ and we are done again.

	If $z_j$ is an edge, it is still true that $C$ is contained in either $\bigcup_{i< j}\Theta_i^{\star}$
	or in $\Theta_j^{\star}$, since otherwise the edge $z_j$ contradicts the
	fact that $C$ is induced in $H$. This concludes the proof that $C$ is contained in a copy $\Theta^{\star}$
	of $\Theta_{k, \ell}$ from $\ccH$ and establishes part~\ref{iti:21b}.

	Before addressing the uniqueness in part~\ref{iti:21c}, we would like to argue that $H$
	is $\Theta_{k, \ell+1}$-free.
	Due to $e(\Theta_{k, \ell+1})=k(\ell+1)\le n$ we could otherwise
	find a forest of copies ${\{\Theta_1^{\star}, \dots, \Theta_m^{\star}\}\subseteq \ccH}$ whose union
	contains a copy~$\Theta^{\star\star}$ of~$\Theta_{k, \ell+1}$.
	Again we consider the least integer $j$
	with $E(\Theta^{\star\star})\subseteq \bigcup_{i\le j}E(\Theta_i^{\star})$,
	observe $j\in [2, m]$, and look at the set $z_j=V(\Theta_j^{\star})\cap\bigcup_{i<j}V(\Theta_i^{\star})$.
	Since $\Theta_{k, \ell+1}$ is $2$-connected, the size of $z_j$ needs to be at
	least two, wherefore $z_j$ is an edge which $\bigcup_{i<j}\Theta_i^{\star}$ and~$\Theta_j^{\star}$
	have in common.

	Since the deletion of any two adjacent vertices from $\Theta_{k, \ell+1}$
	yields a connected graph,~$z_j$ cannot belong
	to $E(\Theta^{\star\star})$, since this would force the copy $\Theta^{\star\star}$ of $\Theta_{k, \ell+1}$ to be contained in the copy
	$\Theta_j^{\star}$ of $\Theta_{k, \ell}$, which is impossible. Consequently, the edge $z_j$ witnesses
	that $\Theta^{\star\star}$ is a non-induced subgraph of~$H$.
	But this causes $H$ to contain a cycle whose length is at most~$k+1$, contrary
	to the already established part~\ref{iti:21b}.

	Let us finally address part~\ref{iti:21c}. Given an induced cycle $C$
	in $H$ of length at most $s$ we already know that $C$ has length $2k$ and
	is contained in some copy~$\Theta^\star$ of $\Theta_{k, \ell}$ in $H$.
	Thanks to our vertex orderings, the vertices $a_0$ and $a_k$ of this copy need
	to be determined by $a_0=\min V(C)$ and $a_k=\max V(C)$.
	Now $\Theta^\star$ consists of $\ell$ internally vertex disjoint ascending
	path from $a_0$ to $a_k$, two of which form $C$. It suffices to show that there is
	no further ascending path from $a_0$ to $a_k$ of length $k$. If such a path~$P$
	existed, then none of its inner vertices could be in $V(\Theta^\star)$,
	because $H$ has girth $2k$. But this means that~$P$ and~$\Theta^\star$
	form a copy of $\Theta_{k, \ell+1}$ in $H$, which is absurd.
\end{proof}

\section{Proof of the main theorem}
\label{sec:proof}
When proving Theorem~\ref{thm:B-set} we sometimes need to know that if two
short sums of elements of our set $X$ agree, then under certain conditions
their summands agree. In all
cases this will ultimately boil down to the following basic arithmetic principle.

\begin{fact}\label{f:trivial}
	If $a_0+a_1m+\dots+a_q m^q=0$ holds for some integers $m\ge 2$, $q\ge 1$,
	and $a_0, \dots, a_q\in (-m, m)$, then $a_0=\dots=a_q=0$.
\end{fact}

\begin{proof}
	Otherwise there exists a least index $i$ such that $a_i\ne 0$. Since
	$a_im^i+\dots+a_qm^q=0$ is divisible by $m^{i+1}$, the coefficient $a_i$
	needs to be divisible by $m$. Together with $|a_i|<m$ this yields the
	contradiction $a_i=0$.
\end{proof}

\begin{proof}[Proof of Theorem~\ref{thm:B-set}]
	For two given numbers $k$, $\l\geq 2$
	we apply Proposition~\ref{prop:21} to $s=2k$, thereby obtaining an
	infinite
	graph $G$ with $V(G)\subseteq\NN$ such that
	\begin{enumerate}[label=\nlabel]
		\item\label{it:n1} $G\lra (\Theta_{k, \ell})_r$ holds for every positive
		      integers $r$,
		\item\label{it:n2} all cycles in $G$ have length at least $2k$,
		\item\label{it:n3} and every $2k$-cycle in $G$ is contained in a unique copy
		      of $\Theta_{k, \ell}$.
	\end{enumerate}
	By relabelling the vertices of $G$ we may assume that they are powers
	of $m=2k+1$, i.e.,
	\[
		V(G)=\{m^n\colon n\in\NN\}.
	\]
	We shall prove that the set
	\[
		X=\{b-a\colon ab\in E(G) \text{ and } a<b\}
	\]
	has all desired properties~\ref{iti:B-set-EN}\,--\,\ref{iti:distinct} of Theorem~\ref{thm:B-set}. Most claims on short sums of elements of $X$ will
	follow from the following statement.

	\begin{clm}\label{clm:31}
		If for two positive integers $s$, $t$ with $s+t\le 2k$ there are elements
		$x_1\le\dots\le x_s$ and $y_1\le \dots\le y_t$ of $X$ with
		$x_1+\dots+x_s=y_1+\dots+y_t$, then one of the following is true.
		\begin{enumerate}[label=\alabel]
			\item\label{it:clma} Either $s=t$ and $(x_1, \dots, x_s)=(y_1, \dots, y_t)$
			\item\label{it:clmb} or $s=t=k$ and there is a unique copy
			      of $\Theta_{k, \ell}$ in $G$, say with vertices $a_0$, $a_k$, and~$a^{(j)}_i$,
			      and there are unique distinct indices $j, j'\in [\ell]$
			      such that $x_i=a^{(j)}_i-a^{(j)}_{i-1}$ and $y_i=a^{(j')}_i-a^{(j')}_{i-1}$
			      holds for all $i\in [k]$ (where, as usual, $a^{(j)}_0=a^{(j')}_0=a_0$
			      and $a^{(j)}_k=a^{(j')}_k=a_k$) and the $2k$ terms $x_1,\dots,x_k,y_1\dots,y_k$ are all distinct.
		\end{enumerate}
	\end{clm}

	\begin{proof}
		We may assume
		\begin{equation}\label{eq:0257}
			\{x_1, \dots, x_s\}\cap \{y_1, \dots, y_t\}=\varnothing\,,
		\end{equation}
		because otherwise we can delete two identical terms from the given equation
		and apply induction on $s+t$, thus ending up in case~\ref{it:clma}.

		Since every integer can be expressed in at most one way as a difference of
		two powers of~$m$, for every $x\in X$ there exists a unique edge $e(x)$ of $G$ with
		\[
			x=\max e(x)-\min e(x)\,.
		\]
		Setting
		\[
			f(n)=\big|\{i\in [s]\colon \max e(x_i)=m^n\}\big|-\big|\{i\in [s]\colon \min e(x_i)=m^n\}\big|
		\]
		for every positive integer $n$, we have $|f(n)|\le s$ and
		\[
			x_1+\dots+x_s=\sum_{n=1}^\infty f(n)m^n\,,
		\]
		where the sum on the right side has only finitely many nonzero terms.
		Similarly, we write
		\[
			y_1+\dots+y_t=\sum_{n=1}^\infty g(n)m^n\,,
		\]
		where the function $g$ is defined analogously by
		\[
			g(n)=\big|\{i\in [t]\colon \max e(y_i)=m^n\}\big|-\big|\{i\in [t]\colon \min e(y_i)=m^n\}\big|
		\]
		and satisfies $|g(n)|\le t$.
		Now the given equality $x_1+\dots+x_s=y_1+\dots+y_t$ entails
		\[
			\sum_{n=1}^\infty \bigl(f(n)-g(n)\bigr)m^n=0\,.
		\]
		Since  for every $n\in\NN$ we have $|f(n)-g(n)|\le|f(n)|+|g(n)|\le s+t\le 2k<m$,
		Fact~\ref{f:trivial} tells us that $f=g$.

		We consider the graph $M\subseteq G$ with edge set
		$E(M)=\{e(x_1), \dots, e(x_s), e(y_1), \dots, e(y_t)\}$.
		If $M$ had a vertex $m^n$ of degree one,
		then $f(n)=g(n)$ would yield a contradiction to~\eqref{eq:0257}.
		Thus $M$
		contains a cycle $C$ and by~\ref{it:n2} the length of $C$ is at least $2k$.
		On the other hand, we know $e(C)\le e(M)=s+t\le 2k$,
		and thus we must have equality throughout. In particular, the $2k$ terms
		$x_1,\dots,x_s,y_1,\dots,y_t$ are distinct.

		By~\ref{it:n3} there is a unique copy
		of $\Theta_{k, \ell}$ in $G$ containing $C$. Denote the vertices of this
		copy in the usual manner by $a_0$, $a_k$, and $a^{(j)}_i$,
		where $(i, j)\in [k-1]\times [\ell]$. Invoking $f=g$ once more, one easily sees
		$s=t=k$ and that there exist distinct indices
		$j, j'\in [\ell]$ such that $\{e(x_1), \dots, e(x_k)\}$
		and $\{e(y_1), \dots, e(y_k)\}$ are the edge sets of the
		paths
		\[
			a_0a^{(j)}_1\cdots a^{(j)}_{k-1}a_k
			\qand
			a_0a^{(j')}_1\cdots a^{(j')}_{k-1}a_k\,,
		\]
		respectively.
		Since the numbers $a_0<a^{(j)}_1<\dots<a^{(j)}_{k-1}<a_k$ are powers of $m$,
		their consecutive differences $a^{(j)}_1-a_0, \dots, a_k-a^{(j)}_{k-1}$ form
		a strictly increasing sequence and thus we have $x_i=a^{(j)}_i-a^{(j)}_{i-1}$
		for every $i\in [k]$. The same argument also yields $y_i=a^{(j')}_i-a^{(j')}_{i-1}$.

		We have thereby found a copy of $\Theta_{k, \ell}$ such that the
		summands $x_1, \dots, x_k$ and $y_1, \dots, y_k$ can be expressed as in
		case~\ref{it:clmb} of the claim. It remains to observe that these summands
		determine the cycle $C$ and, therefore, the copy of $\Theta_{k, \ell}$
		in a unique manner.
	\end{proof}

	Claim~\ref{clm:31} clearly implies the clauses~\ref{iti:Bh}\,--\,\ref{iti:distinct}
	of Theorem~\ref{thm:B-set}. Moreover, since $G$ needs to contain at least one
	copy of $\Theta_{k, \ell}$, there is at least one natural number $n$ satisfying
	$\rho_{X,k}(n)=\ell$ and thus $X$ is indeed a $B_{k, \ell}$-set. The reason
	why the partition property~\ref{it:n1} leads to clause~\ref{iti:B-set-EN} was
	already alluded to in the previous section and altogether it only remains to
	address~\ref{iti:B-set-AE}.

	Let us recall to this end that the map $ab\longmapsto |a-b|$ establishes
	a bijection between~$E(G)$ and~$X$. Combined with Claim~\ref{clm:31} this
	reduces our task to the verification of the following statement:

	\begin{quotation}
		\it
		For every finite set $E'\subseteq E(G)$ there is a subset $E''\subseteq E'$
		of size ${|E''|\ge\frac{k-1}{2k}|E'|}$ such that the graph $(\NN, E'')$ contains
		no ascending path of length $k$.
	\end{quotation}

	Given $E'$ a standard averaging argument yields a partition of $\NN$ into $k$
	classes $V_1, \dots, V_k$ such that the subset $P\subseteq E'$ consisting of all
	edges connecting vertices from distinct classes satisfies $|P|\ge\frac{k-1}k |E'|$.
	We subdivide $P$ further into the two sets
	\[
		P'=\{xy\in P\colon x<y \text{ and there are } i<j
		\text{ with } x\in V_i, y\in V_j\}
	\]
	and
	\[
		P''=\{xy\in P\colon x<y \text{ and there are } i>j
		\text{ with } x\in V_i, y\in V_j\}\,.
	\]
	Neither of the graphs $(\NN, P')$ or $(\NN, P'')$ can contain an ascending path
	of length $k$. Moreover, due to $|P'|+|P''|=|P|\ge \frac{k-1}k |E'|$ we have
	$\max\{|P'|, |P''|\}\ge \frac{k-1}{2k}|E'|$ and, hence, at least one of the
	choices $E''=P'$ or $E''=P''$ is permissible. Thereby Theorem~\ref{thm:B-set}
	is proved.
\end{proof}

\end{document}